\documentclass[a4paper,10pt]{amsart}
\usepackage[utf8]{inputenc}
\usepackage[english]{babel}

\usepackage{amsfonts}

\usepackage[english]{babel}
\usepackage{amsmath,amssymb,verbatim,mathrsfs,latexsym,paralist}
\usepackage{graphicx}
\usepackage{epstopdf}
\usepackage{indentfirst}
\usepackage{multirow}
\usepackage{amsthm}
\usepackage{longtable}
\allowdisplaybreaks[4]

\newcommand{\Z}{\mathbb{Z}}
\newcommand{\C}{\mathbb{C}}

\newcommand{\Q}{\mathbb{Q}}

\newcommand{\GL}{\mathrm{GL}}

\newcommand{\Lie}{\mathrm{Lie}}

\newcommand{\mf}{\mathfrak}
\newcommand{\g}{\mf{g}}
\newcommand{\h}{\mf{h}}
\renewcommand{\a}{\mf{a}}

\newcommand{\gl}{\mf{gl}}

\renewcommand{\u}{\mf{u}}

\newcommand{\z}{\mf{z}}
\newcommand{\ttt}{\mf{t}}

\newcommand{\diag}{\mathrm{diag}}

\newcommand{\SmallMatrix}[1]{\text{{\tiny\arraycolsep=0.4\arraycolsep\ensuremath
    {\begin{pmatrix}#1\end{pmatrix}}}}}

\numberwithin{equation}{section}
\newtheorem{theorem}{Theorem}[section]

\newtheorem{lemma}[theorem]{Lemma}

\theoremstyle{remark}
 
\theoremstyle{remark}
\newtheorem{rmk}[theorem]{Remark}

\title[Computing the Zariski closure of a f.g. matrix group]{Computing the Zariski closure of a finitely generated matrix group}
\author{Willem A. de Graaf}
\address{
Dipartimento di Matematica\\
Universit\`{a} di Trento\\
Italy}
\date{}

\begin{document}

\begin{abstract}
  We describe an algorithm for determining the algebraic subgroup of $\GL(n,\C)$
  that is defined as the closure of the group generated by a finite number of
  elements of $\GL(n,\C)$. The algorithm avoids the use of Gr\"obner bases
  and can be used on non-trivial examples. In the last section we report on
  an implementation of the algorithm in the computer algebra system
  {\tt OSCAR}.
\end{abstract}

\maketitle

\section{Introduction}

By {\em linear algebraic group} (or just algebraic group) we
mean a subgroup of $\GL(n,\C)$ given by the vanishing of a set of polynomials
in $\C[x_{11},x_{12},\ldots,x_{nn}]$. If $G\subset \GL(n,\C)$ is an algebraic
group then by $\Lie(G)$ we denote its Lie algebra, which is a subalgebra
of the Lie algebra $\gl(n,\C)$ consisting of all $n\times n$-matrices with
coefficients in $\C$. We refer to the standard literature for accounts of
the basic properties of algebraic groups and their Lie algebras
(\cite{borel,hum2,tauvelyu}). 

Let $g_1,\ldots,g_s$ be invertible $n\times n$-matrices with coefficients in
$\C$ (in other words, they are elements of $\GL(n,\C)$). Let $\mathcal{H}$
be the subgroup of $\GL(n,\C)$ generated by them. Let $\overline{\mathcal{H}}$
be the Zariski closure of $\mathcal{H}$ in $\GL(n,\C)$. Then
$\overline{\mathcal{H}}$ is a linear algebraic group
(\cite[Proposition 1.3(b)]{borel}). In this paper we consider the problem
to compute $\overline{\mathcal{H}}$ given $g_1,\ldots,g_s$. In order to be
able to work with these elements on a computer we assume that the entries
lie in an extension of finite degree of $\Q$.

This problem has been considered in various ways in the literature.
A first general algorithm is developed in \cite{derksenetal}. This paper also
outlines applications in the theory of quantum automata. We refer to 
\cite{bdp} for further developments in that direction. In \cite{gasta} the
authors look at Zariski closures of cyclic matrix groups (i.e., when $s=1$ in
the above description) and show that their irreducible components are toric
varieties. The paper \cite{hoaw18} considers applications in computer science,
and a generalisation of the problem to groups generated by constructible
sets of matrices. 

First we say some words on what it means to ``compute'' an algebraic
subgroup of $\GL(n,\C)$, in other words, how we can represent the output
of our algorithm (i.e., an algebraic group)
by a finite amount of data. To express this we say that
we ``specify'' an algebraic group.

The most obvious way to specify an
algebraic subgroup $G$ of $\GL(n,\C)$ is by giving a finite set of polynomials
$p_1,\ldots,p_m\in \C[x_{11},x_{12},\ldots,x_{nn}]$ such that $G$ consists of
all $g\in \GL(n,\C)$ such that $p_i(g)=0$ for $1\leq i\leq m$. This is
the most immediate and uncomplicated way to specify an algebraic group. However,
it has a few drawbacks: the set of polynomials can be very bulky, and
given the polynomials it is not easy to obtain information on the group:
whether it is connected or not, solvable or not, what the
dimension of its unipotent radical is, and so on.

A second way to specify $G$ is based on two facts of the theory of algebraic
groups. Firstly, $G$ is a finite union of (pairwise disjoint) components
(\cite[Propositon 1.2(b)]{borel}). 
The component containing the identity is denoted $G^\circ$ and is a normal
subgroup. The other components are the cosets of $G^\circ$ in $G$, so are of
the form $gG^\circ$ for an element $g\in G$. 
A second fact says that the Lie algebra of $G$, which is the same as
the Lie algebra of $G^\circ$, completely determines the identity component
$G^\circ$ (\cite[\S 7.1(2)]{borel}). So we can specify $G$ by giving 
a basis of its Lie algebra, as subalgebra
of $\gl(n,\C)$, and one element of each component of $G$. The main advantage
of this approach is that many questions on $G$ can immediately be answered
by considering the analogous question for the Lie algebra. A disadvantage is
that it is only applicable in characteristic zero. Secondly, it is not
immediate to decide whether a given $g\in \GL(n,\C)$ lies in $G$. In Section
\ref{sec:in} we will give an algorithm for this.

We remark that from a specification of an algebraic group given by polynomials
it is possible to compute its Lie algebra and its components. However, to
determine the Lie algebra in general one needs to determine the radical
of the ideal generated by the polynomials. For the components one needs to
compute the primary decomposition of the same ideal. There are algorithms for
both tasks (\cite[Chapter 8]{beckwsp}), but they rely on heavy Gr\"obner
basis computations and hence can be very difficult to carry out in practice.

In \cite{derksenetal} algorithms are given for the main problem of this paper
that work for
groups in any characteristic. The output is specified using polynomials and one
of the main workhorses of the algorithm are elimination procedures based
on Gr\"obner bases. The latter make the algorithms difficult to carry out
in practice. In fact, the authors write in \cite{derksenetal}
`` In characteristic 0, one might replace H by its tangent space at the
identity. The algorithm should then be modified accordingly. This way one
might avoid Gr\"obner basis computations in the algorithm and one might end
up with an algorithm that is actually practical.'' 

In this paper we essentially do what is written in this last citation.
We use the second method above to specify an algebraic group; that is, by a
basis of its Lie algebra and one element of each component. We describe an
algorithm that does not use Gr\"obner basis computations. In certain steps
it uses computations in number fields, namely, for diagonalizing a torus
(see Sections \ref{sec:tori}, \ref{sec:in}) and for computing the 
multiplicative relations of the eigenvalues of a matrix (see Section
\ref{sec:lieGg}). The main
bottleneck of the algorithm lies now in those computations; with growing
degree of the needed field extensions the computations become rapidly more
difficult. However, we can use the algorithm to deal with nontrivial
examples, see Section \ref{sec:exa}.

We start with a section on algorithms to work with tori (Section
\ref{sec:tori}). In the subsequent Section \ref{sec:in} we give an
algorithm for deciding whether a given $g\in \GL(n,\C)$ lies in a connected
algebraic group $G\subset \GL(n,\C)$, given by a basis of its Lie algebra.
In Section \ref{sec:lieGg} we consider the problem to compute a basis of the
Lie algebra of the smallest algebraic group containing a given $g\in \GL(n,\C)$.
Section \ref{sec:main} has the main algorithm of the paper.
In the last section we report on an implementation of the algorithm in the
computer algebra system {\tt OSCAR} (\cite{OSCAR,OSCAR-book}).

\vspace{2mm}

\noindent{\bf Acknowledgement.} I am very grateful to Thomas Breuer,
Claus Fieker and Max Horn for stimulating discussions and their help regarding
the system {\tt OSCAR}. I also thank Mima Stanojkovski for her remarks on a
previous version of the manuscript. 

\section{Lattices, tori and their Lie algebras}\label{sec:tori}

By a {\em lattice} we mean a finitely generated subgroup of
$\Z^n$ for some $n\geq 1$. For a general introduction into such groups
and to methods for computing with them we refer to the book by Sims,
\cite[Chapter 8]{sims}. A lattice $\Lambda\subset \Z^n$ has a basis, which is
a set of elements $u_1,\ldots,u_k\in \Z^n$ such that each $u\in \Lambda$ can
uniquely be written as $u=m_1u_1 +\cdots +m_k u_k$ with $m_i\in \Z$ for
$1\leq i\leq k$. A lattice $\Lambda$ is called {\em pure} if $\Z^n/\Lambda$ is
torsion free, or equivalently, if for $u\in \Z^n$, $m\in \Z$ we have that
$mu\in \Lambda$ if and only if $u\in \Lambda$.

By definition a {\em torus} in $\GL(n,\C)$ is a connected algebraic
subgroup consisting of commuting semisimple elements (\cite[\S 8.5]{borel}).
An example of a torus is the group
$$D(n,\C) = \{ \diag(a_1,\ldots,a_n) \mid a_1\cdots a_n\neq 0\}.$$
If $T\subset \GL(n,\C)$ is a torus then there is a $C\in \GL(n,\C)$ such
that $CTC^{-1} \subset D(n,\C)$. If we have a basis of the Lie algebra
$\ttt= \Lie(T)$ of $T$ then we can find such a $C$. Indeed, $\ttt$ also
consists of commuting semisimple elements so by simultaneously diagonalizing
the elements of a basis of $\ttt$ we can find a matrix $C$ such that
$C\ttt C^{-1}$ consists of diagonal matrices. Now
$\mathrm{Int}_C : g\mapsto CgC^{-1}$ is an automorphism of $\GL(n,\C)$ with
differential $\mathrm{Ad}_C : x\mapsto CxC^{-1}$ for $x\in \gl(n,\C)$
(\cite[\S 3.13]{borel}). Define $T' = CTC^{-1}$, then $\mathrm{Int}_C :
T\to T'$ is
an isomorphism, so the same holds for  $\mathrm{Ad}_C : \Lie(T) \to \Lie(T')$.
Hence $\Lie(T') = C\ttt C^{-1}$. Since $T'$ is connected and $\Lie(T')$ consists
of diagonal matrices it follows that $T'\subset D(n,\C)$. 

Let $T\subset D(n,\C)$ be a torus and set $\ttt=\Lie(T)$.
Define the following lattices
\begin{align*}
\Lambda(T) &= \{(e_1,\ldots,e_n)\in \Z^n \mid \prod_{i=1}^n a_i^{e_i}=1
\text{ for all } \diag(a_1,\ldots,a_n)\in T\}\\
\Lambda(\ttt) &= \{(e_1,\ldots,e_n)\in \Z^n \mid \sum_{i=1}^n b_ie_i=0
\text{ for all } \diag(b_1,\ldots,b_n)\in \ttt\}.
\end{align*}

For $e=(e_1,\ldots,e_n)$ let $\chi_e$ be
the character of $D(n,\C)$ given by
$$\chi_e(\diag(a_1,\ldots,a_n))=\prod_i a_i^{e_i}.$$
Its differential is given by
$d\chi_e (\diag(b_1,\ldots,b_n))=\sum_i e_ib_i$.

\begin{lemma}\label{lem:tor1}
$\Lambda(T)=\Lambda(\ttt)$.
\end{lemma}  
  
\begin{proof}
Let $e\in \Lambda(T)$ then $\chi_e$ maps $T$ to $1$, hence
$d\chi_e$ maps $\ttt$ to $0$, so that $e\in \Lambda(\ttt)$. For the converse let
$e\in \Lambda(\ttt)$. This means that 
$d \chi_e$ maps $\ttt$ to $0$. In other words, $\Lie (\chi_e(T))=0$. Hence
$\chi_e(T)$ is a finite group so there
is a $k>0$ such that for all $g\in T$ we have $\chi_e(g)^k=1$. But
$\chi_e(g)^k = \chi_{ke}(g)$. It follows that $ke\in \Lambda(T)$. As $T$ is
connected, $\Lambda(T)$ is pure (\cite[Proposition 3.9.7]{gra16}).
It follows that $e\in \Lambda(T)$.
\end{proof}

For a lattice $\Lambda\subset \Z^n$ we define
\begin{align*}
T(\Lambda) &= \{ \diag(a_1,\ldots,a_n)\mid \prod_{i=1}^n a_i^{e_i}=1
\text{ for all } (e_1,\ldots,e_n)\in \Lambda\}\\
\ttt(\Lambda) &= \{ \diag(b_1,\ldots,b_n) \mid \sum_{i=1}^n b_ie_i=0
\text{ for all } (e_1,\ldots,e_n)\in \Lambda\}.
\end{align*}

\begin{lemma}\label{lem:tor2}
Let $T\subset   D(n,\C)$, $\ttt=\Lie(T)$ be as above. Set $\Lambda=\Lambda(T)=
\Lambda(\ttt)$ (cf. Lemma \ref{lem:tor1}). Then $T(\Lambda)=T$ and
$\ttt(\Lambda)=\ttt$.
\end{lemma}

\begin{proof}
A well known theorem states that an algebraic subgroup of $D(n,\C)$ is the
intersection of the kernels of a number of characters $\chi_e$ (cf.
\cite[Proposition 8.2(c)]{borel}).
Since $\Lambda$ is the largest set of $e\in \Z^n$ such that $T$ is contained
in the kernel of $\chi_e$ for all $e\in \Lambda$, the expression for $T$
follows. Because $\Lie(T(\Lambda)) = \ttt(\Lambda)$ this implies the formula
for $\ttt$ as well.
\end{proof}

Let, as before, $T\subset D(n,\C)$ be a torus and $\ttt=\Lie(T)$ its Lie
algebra. By the previous lemma $\ttt = \ttt(\Lambda)$, where $\Lambda=
\Lambda(\ttt)$. So $\ttt$ has a basis consisting of matrices with entries
in $\Q$.
Given such a basis $B$ we can compute a basis of the vector space
$$M_\ttt=\{ (\gamma_1,\ldots,\gamma_n)\in \Q^n \mid \sum_{i=1}^n \gamma_i b_i =0
\text{ for all } \diag(b_1,\ldots,b_n)\in B\}.$$
Furthermore, $\Lambda(\ttt) = M_\ttt\cap \Z^n$, which can be computed by 
the ``purification algorithm'' (\cite[\S 6.2]{gra16}). We conclude that,
given a basis of $\ttt$ we can compute a basis of $\Lambda(\ttt)$.

\begin{rmk}\label{rem:intor}
If $T$ is a torus and $s\in \GL(n,\C)$ semisimple then we can test whether
$s\in T$, just knowing a basis of $\ttt=\Lie(T)$. Indeed, as seen above
we can compute a matrix $C\in \GL(n,\C)$ such that $T' = CT C^{-1}$ is
contained in $D(n,\C)$.
Set $\ttt'=C\ttt C^{-1}$, then $\ttt'=\Lie(T')$.  As outlined above we can
compute a basis of $\Lambda=\Lambda(\ttt')= \Lambda(T')$. As $T'=T(\Lambda)$
we see that $s'=CsC^{-1}$ lies in $T'$ if and only if we have $s'=\diag(a_1,
\ldots,a_n)$ and
$$\prod_{i=1}^n a_i^{e_i}=1 \text{ for all } (e_1,\ldots,e_n)\in \Lambda.$$
In fact, it suffices to check the latter condition for the elements of
a basis of $\Lambda$. Finally we remark that $CsC^{-1}\in T'$ is equivalent
to $s\in T$.
\end{rmk}  

\section{Deciding membership}\label{sec:in}

Let $G\subset \GL(n,\C)$ be a connected algebraic group, $\g=\Lie(G)\subset
\gl(n,\C)$ its Lie algebra. Let $g\in \GL(n,\C)$ be given. In this section
we consider the problem to decide whether $g\in G$, just knowing a basis of
$\g$.

A first step is to compute the multiplicative Jordan decomposition
$g=su$ ($s$ semisimple, $u$ unipotent). As algebraic groups are closed
under this decomposition (\cite[Theorem 4.4]{borel}) we have that 
$g\in G$ if and only if both $s,u\in G$. Furthermore, it is easy to decide
whether $u\in G$. Indeed, as $u$ is unipotent we have that $u-1$ is nilpotent
and we define
$$\log(u) = \sum_{i\geq 1} (-1)^{i-1}\frac{(u-1)^i}{i}.$$
Let $G(u)$ be the smallest algebraic subgroup of $\GL(n,\C)$ containing $u$.
Then $G(u)$ is of dimension 1 and its Lie algebra is spanned by
$\log(u)$ (\cite[Proposition 4.3.10]{gra16}). We have $u\in G$ if and only if 
$G(u)\subset G$, which, since both groups are connected, is equivalent
to $\Lie(G(u))\subset \g$ (\cite[\S 7.1]{borel}). We conclude that
$u\in G$ if and only if $\log(u)\in \g$. The latter condition can easily be
tested. So it remains to decide whether $s\in G$.

By definition a
{\em Cartan subgroup} of $G$ is the centralizer of a maximal torus of $G$.
We have the following facts:
\begin{itemize}
\item A Cartan subgroup of $G$ is connected (\cite[Corollary 11.12]{borel}).
\item Let $T$ be a maximal torus of $G$ and $H=Z_G(T)$ the corresponding
  Cartan subgroup. Then $H=T\times H_u$, where $H_u$ is a subgroup of
  $H$ consisting of all unipotent elements of $H$
  (\cite[Theorem 12.1(d)]{borel}).
\item A subalgebra $\h\subset \g$ is a Cartan subalgebra if and only if it
  is the Lie algebra of a Cartan subgroup (\cite[Proposition 29.2.5]{tauvelyu}).
\end{itemize}

In particular it follows that if $\h$ is a Cartan subalgebra of $\g$
then $\h =\ttt\oplus \u$ where $\ttt$ consists of semisimple elements and
$\u$ consists of nilpotent elements. With the above notation we have
$\ttt=\Lie(T)$ and $\u=\Lie(H_u)$. 
We can compute bases of $\ttt$ and $\u$
by computing the Jordan decomposition of all basis elements of $\h$. Then
$\ttt$ is spanned by all semisimple parts and $\u$ is spanned by all nilpotent
parts. Indeed, the set consisting of all semisimple and nilpotent parts spans
$\h$, and the semisimple parts lie in $\ttt$ and the nilpotent parts lie
in $\u$.

\begin{theorem}\label{thm:sint}
Let $G$, $\g$ be as above. Let $s\in \GL(n,\C)$ be semisimple and suppose that
$s\g s^{-1} = \g$. Let $\z_\g(s) = \{ x\in \g\mid sx=xs\}$ and let $\h$ be a
Cartan subalgebra of $\z_\g(s)$. Write $\h=\ttt\oplus \u$, where $\ttt$ is the
set of semisimple elements of $\h$ and $\u$ is the set of nilpotent elements
of $\h$. Let $T\subset \GL(n,\C)$ be the connected subgroup with $\Lie(T) =
\ttt$. Then $s\in G$ if and only if $s\in T$.
\end{theorem}

\begin{proof}
Set $Z_G(s) = \{ g\in G \mid gs=sg\}$. Then $Z_G(s)$ is given, as subgroup of
$G$ by a set of linear equations. Therefore its Lie algebra is given, as a
subalgebra of $\g$, by the same equations. Hence $\z_\g(s)$ if the Lie algebra
of $Z_G(s)$. In particular $\z_\g(s)$ is the Lie
algebra of an algebraic group, and therefore $\h$ has the stated
decomposition. Furthermore, $\h$ is the Lie algebra of a Cartan subgroup
which is the centralizer of a maximal torus $T'$ of $Z_G(s)^\circ$. The Lie
algebra of $T'$ is $\ttt$, and therefore $\ttt$ is algebraic and $T'=T$.

Now suppose that $s\in G$. By \cite[Proposition 28.3.1(iii)]{tauvelyu}
$s\in  Z_G(s)^\circ$.
Since $\h$ is a Cartan subalgebra of
$\z_\g(s)$ it follows that $T$ is a maximal torus of $ Z_G(s)^\circ$.
The centralizer of $T$ in $Z_G(s)^\circ$ is the Cartan subgroup $H$ with
$\Lie(H)=\h$. Write $H=T\times H_u$, where $T$ is the set of semisimple elements
of $H$ and $H_u$ is the set of unipotent elements of $H$. Since $s$ commutes
with all elements of $Z_G(s)^\circ$ it follows that $s\in H$ so that $s\in T$.
The other implication is trivial. 
\end{proof}

We recall that for $g\in G$ the adjoint map $\mathrm{Ad}(g)$, mapping
$x\in \g$ to $gxg^{-1}$ is an automorphism of $\g$ (\cite[\S 3.13]{borel}).
So if $s\g s^{-1}\neq \g$ then $s\not\in G$.

Now the procedure for checking whether $s\in G$ is straightforward:
\begin{enumerate}
\item If $s\g s^{-1} \neq \g$ then $s\not\in G$, otherwise
  continue.
\item Compute a basis of $\z_\g(s)$.
\item Compute a Cartan subalgebra $\h$ of $\z_\g(s)$ (we refer to
  \cite[\S 3.2]{gra6} for algorithms for this task)
  and bases of $\ttt$,
  $\u$ where $\ttt$ is the set of semisimple elements of $\h$ and $\u$ is
  the set of nilpotent elements of $\h$.
\item Let $T\subset G$ be the connected algebraic group with Lie algebra
  $\ttt$. Use the procedure outlined in Remark \ref{rem:intor} for checking
  whether $s\in T$. By Theorem \ref{thm:sint} we have $s\in T$ if and only if
  $s\in G$.
\end{enumerate}

\section{Computing the Lie algebra of $G(g)$}\label{sec:lieGg}

For $g\in \GL(n,\C)$ we denote by $G(g)$ the smallest algebraic subgroup of
$\GL(n,\C)$ containing $g$. Here we consider computing a basis of the Lie
algebra of $G(g)$ when $g$ is either semisimple or unipotent.

The unipotent case is straightforward. 
If $u\in \GL(n,\C)$ is unipotent then $G(u)$ is 1-dimensional and
$\Lie (G(u))$ is spanned by $\log(u)$ (\cite[Proposition 4.3.10]{gra16}). 

Let $s\in \GL(n,\C)$ be semisimple. Let $C\in \GL(n,\C)$ be such that
$s'= CsC^{-1}$ is diagonal, $s'=\diag(\alpha_1,\ldots,\alpha_n)$. Set
$$\Lambda'= \{ e\in \Z^n \mid \prod_{i=1}^n \alpha_i^{e_i}=1\}.$$
As $s'\in D(n,\C)$ we also have $G(s')\subset D(n,\C)$. So by
\cite[Proposition 8.2(c)]{borel} $G(s')$ is the intersection of a number of
characters of $D(n,\C)$. These characters are of the form $\chi_e$ for
$e\in \Z^n$ (notation as in Section \ref{sec:tori}) and $s'$ has to lie in
their kernel. It follows that $G(s')=T(\Lambda')$.
Hence $\Lie (G(s')) = \ttt(\Lambda')$. (Note, however, that $G(s')$ is not
necessarily connected; this is the case if and only if $\Lambda'$ is pure.)
Therefore $\Lie (G(s)) = C^{-1}\ttt(\Lambda') C$. 

So the remaining problem is to compute a basis of $\Lambda'$.
This can be done by an algorithm due to Ge, \cite{ge_thesis}.
Recently Combot (\cite{combot}) has developed a different approach.
We refer to the references in the latter paper for an overview of the
literature on this subject. Here we also mention the paper \cite{zheng},
which, among other things, has a sufficient criterion for establishing that
the lattice of multiplicative relations is trivial (this is for example the
case when the Galois group is 2-transitive).

\section{Computing the Zariski closure}\label{sec:main}

Let $g_1,\ldots,g_s\in \GL(n,\C)$. We assume that they have coefficients in
$\Q$, or more generally in a finite extension of $\Q$. In this section we
describe an algorithm to compute the Zariski closure $G$ of the group
generated by the $g_i$.
The output is a basis of the Lie algebra $\g\subset \gl(n,\C)$ of $G$ along
with an element of each component of $G$.

The main idea is to approximate the Lie algebra $\g$ by a Lie algebra
$\hat \g\subset \g$. If $\hat \g$ happens to be equal to $\g$ then we
know the identity component $G^\circ$. Note that the group $U$ generated
by $g_1,\ldots,g_s$ is dense in $G$, so $U$ meets every component of $G$.
Therefore $g_1G^\circ,\ldots,g_sG^\circ$ generate $G/G^\circ$. So by enumerating
products of the $g_iG^\circ$ we obtain the component group. If $\hat \g$ is
strictly smaller than $\g$ we let $\widehat{G}$ be the connected algebraic
group with Lie algebra $\hat \g$. We enumerate products of elements
$g_i\widehat{G}$. When doing that, at some point we find an element
$h\widehat{G}$ such that $\Lie( G(h))$ is not contained in $\hat \g$.
When that happens we enlarge $\hat \g$. 

By computing multiplicative Jordan decompositions we may assume that
each $g_i$ is either semisimple or unipotent. So the input to the
algorithm is a set $A$ of elements of $\GL(n,\C)$ that are either semisimple
or nilpotent. The algorithm takes the
following steps.

\begin{enumerate}  
\item Let $\hat\g$ be the Lie algebra generated by all $\Lie (G(g))$ for
  $g\in A$.
\item While there is $g\in A$ with $g\hat\g g^{-1} \neq \hat\g$ replace
  $\hat\g$ by
  the Lie algebra generated by $\hat\g$ and $g\hat\g g^{-1}$.
\item Let $C=\emptyset$. Let  $\widehat{G}$ be the
    connected algebraic group with Lie algebra $\hat\g$.
  For $l=1,2,\ldots$ compute all products of elements of $A$ of length
  $l$. For each new element $a$ do the following:
  \begin{enumerate}
  \item If there is a $b\in C$ with $ab^{-1}\in \widehat{G}$ (to decide this we
    use the algorithm of Section \ref{sec:in}) then
    discard $a$.
  \item Compute the Jordan decomposition $a=su$ and $\Lie(G(s))$,
    $\Lie(G(u))$. If one of these algebras is not contained in $\hat\g$ then
    add $s,u$ to $A$ and return to (1).
  \item Add $a$ to the set $C$.
  \item If no products of length $l$ have been added to $C$ then stop, and the
    output is $\hat\g$, $C$.
  \end{enumerate}  
\end{enumerate}  

\begin{theorem}
The algorithm terminates and upon termination $\hat\g=\Lie(G)$ and
$C$ contains exactly one element of each component of $G$.  
\end{theorem}

\begin{proof}
We recall that a Lie subalgebra of $\gl(n,\C)$ is said to be algebraic if it
is the Lie algebra of an algebraic subgroup of $\GL(n,\C)$
(\cite[\S 7]{borel}).   
The Lie algebra constructed in Step 1 is generated by algebraic Lie algebras
and hence is algebraic (\cite[Corollary 7.7(d)]{borel}).
The same holds for the Lie algebra constructed in
Step 2. However, that Lie algebra additionally is normalized by all
$g\in A$. Let $\hat\g$ denote the Lie algebra at the end of Step 2, and let
$\widehat{G}$ be the connected algebraic subgroup of $\GL(n,\C)$ with Lie
algebra $\hat\g$. Then $\widehat{G}$ is also normalized by all $g\in A$. Indeed,
the Lie algebra of $g\widehat{G}g^{-1}$ is $g\hat\g g^{-1}$. But in
characteristic zero a connected algebraic group is determined by its Lie algebra
(cf. \cite[\S 7.1]{borel}) so because for $g\in A$ we have
$g\hat\g g^{-1} = \hat \g$ it also holds that $g\widehat{G}g^{-1}=\widehat{G}$.

We claim that $\widehat{G}$ is a normal subgroup of $G$. In order to show
that let $R= \C[x_{11},x_{12},\ldots,x_{nn},\tfrac{1}{\det(x_{ij})}]$.
Let $p_1,\ldots,p_r\in R$
generate the vanishing ideal of $\widehat{G}$. For a fixed $h\in
\widehat{G}$ define $p_i^h\in R$ by $p_i^h(g) = p_i(ghg^{-1})$ for
$g\in \GL(n,\C)$. Let $\mathcal{G}\subset \GL(n,\C)$ be the group
generated by $A$.
Then all elements of $\mathcal{G}$ normalize $\widehat{G}$. So for $h\in
\widehat{G}$ and $1\leq i\leq r$ we have that $p_i^h$ vanishes on
$\mathcal{G}$. Hence it vanishes on
the closure of $\mathcal{G}$ which is $G$. So for $g\in G$ we have that
$p_i(ghg^{-1})=0$ for all $h\in \widehat{G}$ and $1\leq i\leq r$. It follows that
$ghg^{-1}\in\widehat{G}$ for $h\in \widehat{G}$ and $g\in G$. 

Note that after Step 3(b) the dimension of the Lie algebra $\hat\g$
increases. So after
a finite number of rounds this does not happen any more. 
Let $\hat\g$ and $\widehat{G}$ be the Lie algebra
and corresponding group at that point. Write $A=\{g_1,\ldots,g_s\}$ and 
let $Q\subset G/\widehat{G}$ be the
group generated by $g_1\widehat{G},\ldots,g_s\widehat{G}$. Consider products of
the form
\begin{equation}\label{eq:prod}
  g_{i_1}\cdots g_{i_k} \text{ for some $k\geq 1$}.
\end{equation}
In Step (3) all such products are enumerated. Let $a$ be an element of the
form \eqref{eq:prod}. 
At some point this element is considered in Step (3).
As the second half of Step 3(b) is not
executed we have that $\Lie (G(a)) \subset \hat \g$, whence $G(a)^\circ
\subset \widehat{G}$. Since $G(a)/G(a)^\circ$ is finite it follows that $a^r\in
G(a)^\circ$ for some $r\geq 1$. This also means that $a^r\in \widehat{G}$.
So the inverse of $a\widehat{G}$ is also the coset of a product of the form
\eqref{eq:prod}.
Therefore $Q$ consists of the cosets of the products of the form
\eqref{eq:prod} and all elements of $Q$ are of finite order.

By \cite[Theorem 5.6]{borel} there is a linear representation $\rho :
G\to \GL(m,\C)$ with kernel $\widehat{G}$. This induces an injective
homomorphism $\bar \rho : Q\to \GL(m,\C)$. The image of $\bar\rho$ consists
of elements of finite order. It is known that a finitely generated subgroup
of $\GL(m,\C)$ whose elements have finite order, must be finite
\cite[Theorem G of Section II.2]{kaplansky}.
It follows that $Q$ is finite, and hence the algorithm terminates.

Upon termination $\widehat{G}$ is a connected normal subgroup of finite
index in $G$ hence it equals $G^\circ$. Therefore $\hat\g=\Lie(G)$.
\end{proof}

\section{Examples and practical experiences}\label{sec:exa}

The algorithms described in this paper have been implemented\footnote{Currently the implementation is not publicly available, but can be obtained from the author.} in 
the computer algebra system {\tt OSCAR} (\cite{OSCAR-book,OSCAR}).
The core of this system is written in the programming language {\tt Julia}.
However, it also includes and builds on computer algebra systems like
{\tt GAP}, {\tt Singular}, {\tt Polymake} and {\tt ANTIC}. In particular,
in {\tt OSCAR} it is possible to start a {\tt GAP} session, and from that
{\tt GAP} session it is possible to access and work with the functions
written in the {\tt Julia} part.

The main body of the implementation of the algorithms of this paper has been
written in {\tt GAP}. For the problem of computing the multiplicative relations
of the eigenvalues of a matrix we access the {\tt Julia} part, which has
an implementation of an algorithm for that purpose. For the diagonalization
of a torus (which is needed for the algorithm of Section \ref{sec:in}) we
use the algebraic closure of $\Q$, which is implemented in the {\tt Julia}
part of {\tt OSCAR} and to which the {\tt GAP} part has an interface.

We consider example inputs that are constructed in the following way.
Let $\a$ be a simple complex Lie algebra and let $\rho : \a \to \gl(V)$ be a
finite-dimensional representation. Let $\Phi$ denote the root system of $\a$
(with respect to a fixed Cartan subalgebra $\h$) and let $x_{\alpha}$ for
$\alpha\in
\Phi$ be the root vectors lying in a fixed Chevalley basis of $\a$
(cf. \cite[\S 25]{hum}). For $t\in \C$ we write $x_{\alpha}(t) =
\exp(t\rho(x_\alpha))$. Furthermore, for $\alpha\in \Phi$ we set
$$w_\alpha = x_{\alpha}(1)x_{-\alpha}(-1)x_{\alpha}(1).$$
Then $w_\alpha\in \GL(V)$ normalizes $\rho(\h)$ and induces the reflection
$s_\alpha$ on $\h$. Let $\Delta = \{\alpha_1,\ldots,\alpha_\ell\}$ be a fixed
basis of $\Phi$. Then the Weyl group $W$ of $\Phi$ is generated by the
simple reflections $s_i = s_{\alpha_i}$ for $1\leq i\leq \ell$.
For $w\in W$ we fix a reduced expression $w=s_{i_1}\cdots s_{i_k}$ and set
$\dot{w} = w_{\alpha_{i_1}}\cdots w_{\alpha_{i_k}}$. Now we take (randomly) $\ell$
elements $w_1,\ldots,w_\ell$ in $W$ and set
$$g_i = x_{\alpha_i}(3) \dot{w}_i x_{\alpha_i}(-3) \text{ for } 1\leq i\leq \ell.$$
Then $g_1,\ldots,g_\ell$ is the input to our algorithm. These are all of
finite order, so each $\Lie(G(g_i))$ is zero. However, it turns out to be
likely that their closure is the connected group with Lie algebra $\rho(\a)$.
We now give these matrices for a few specific cases.

For $\a$ of type $G_2$ and $V$ of dimension 7 we get
$$\SmallMatrix{0&0&0&0&-1&0&-3\\0&0&0&0&0&0&-1\\
      0&-1&0&6&-18&-9&27\\0&0&0&1&-3&-3&9\\
      0&0&0&0&0&-1&3\\1&-3&3&-18&27&0&0\\
      0&0&1&-6&9&0&0 },\,\,
  \SmallMatrix{0&0&0&0&0&0&-1\\0&0&0&0&-3&10&0\\
      0&0&0&0&1&-3&0\\0&0&0&-1&0&0&0\\
      0&3&10&0&0&0&0\\0&1&3&0&0&0&0\\
      -1&0&0&0&0&0&0 }.$$

For $\a$ of type $A_3$ and $V$ of dimension 6 we get
$$\SmallMatrix{
      0&0&0&0&0&-1\\0&0&-1&0&3&0\\0&0&0&0&1&0\\
      0&-1&3&3&-9&0\\0&0&0&1&-3&0\\-1&0&0&0&0&0},\,\,
\SmallMatrix{0&-3&1&0&0&0\\0&-1&0&0&0&0\\1&-3&0&0&0&0\\
      0&0&0&0&0&1\\0&0&0&3&-1&3\\0&0&0&1&0&0},\,\,
\SmallMatrix{0&0&0&-1&0&0\\-1&-3&0&-9&0&0\\0&0&0&0&-1&-3\\
  0&1&0&3&0&0\\0&0&0&0&0&1\\0&0&-1&0&-3&0}.$$

Now we let $\a$ of type $B_2$ and for $V$ take the direct sum of two copies
of the irreducible module of dimension 4. This yields two elements of
$\GL(8,\C)$. We multiply the second element by the matrix $\SmallMatrix{
  0_4 & I_4 \\ I_4 & 0_4 }$ (where $0_4$, $I_4$ denote, respectively, the
$4\times 4$ zero matrix and the $4\times 4$ identity matrix).
By doing this we get a Zariski closure that has
two components. The matrices turn out to be

$$\SmallMatrix{0&0&-1&0&0&0&0&0\\-1&0&0&-3&0&0&0&0\\
      0&0&0&1&0&0&0&0\\0&-1&-3&0&0&0&0&0\\
      0&0&0&0&0&0&-1&0\\0&0&0&0&-1&0&0&-3\\
      0&0&0&0&0&0&0&1\\0&0&0&0&0&-1&-3&0 },\,\,
\SmallMatrix{0&0&0&0&-3&9&-1&3\\0&0&0&0&-1&3&0&0\\
      0&0&0&0&0&-3&0&1\\0&0&0&0&0&-1&0&0\\
      -3&9&-1&3&0&0&0&0\\-1&3&0&0&0&0&0&0\\
      0&-3&0&1&0&0&0&0\\0&-1&0&0&0&0&0&0}.$$

In Table \ref{fig:1} we give the runtimes of our implementation on these
inputs.  The second to third columns of this table have data relative to the
algorithm for computing multiplicative relations. The second column has the
degree of the field extensions (of the fields that contains the eigenvalues).
The third and fourth column have the number of times this routine was called,
and its total running time. The fourth and fifth column display, respectively,
the number of times the algorithm of Section \ref{sec:in} was called and its
total running time. The last column has the total running time of the
algorithm.

We see that on these inputs the running time is dominated by the time needed
for the algorithm of Section \ref{sec:in}. The time mainly goes into doing
arithmetic with algebraic numbers needed to diagonalize a torus. 

\begin{table}[htb]\label{fig:1}
  \caption{Running times of our algorithm on the inputs described in the text.}
\begin{tabular}{|l|r|r|r|r|r|r|}
\hline
$\a$ & \multicolumn{3}{|c|}{mult. rels} & \multicolumn{2}{|c|}{element test} &
total time (s)\\
\hline
& degree & count & time (s) & count & time (s) & \\
\hline
$G_2$ & 12 & 14 & 0.4 & 23 & 859 & 959 \\
\hline
$B_2$ & 8 & 20 & 46 & 105  & 105 & 157 \\
\hline
$A_3$ & 24 & 19 & 1.2 & 34 & 3699 & 4618 \\
\hline
\end{tabular}
\end{table}

\end{document}